\documentclass[a4paper,oneside]{article}
\usepackage[utf8]{inputenc}
\usepackage[a4paper,margin=3cm]{geometry} 
\usepackage{amsmath}
\usepackage{amsthm}
\usepackage{amscd}
\usepackage{amssymb}
\usepackage{MnSymbol}
\usepackage{latexsym}
\usepackage{eucal}
\usepackage{dsfont}
\usepackage{mathtools}
\usepackage{enumitem}
\usepackage{verbatim}

\usepackage{pdflscape}

\usepackage[colorlinks,pdftex]{hyperref}
\hypersetup{
linkcolor=black,
citecolor=black,
pdftitle={},
pdfauthor={Franziska K\"uhn},
pdfkeywords={},
}

\makeatletter
\renewcommand\section{\@startsection{section}{1}{0mm}{-1.5\baselineskip}{\baselineskip}{\normalsize\bfseries\sffamily}}
\renewcommand\subsection{\@startsection{subsection}{1}{0mm}{-\baselineskip}{\baselineskip}{\normalsize\bfseries\sffamily}}
\makeatother

\makeatletter
\def\@fnsymbol#1{\ensuremath{\ifcase#1\or *\or **\or \dagger\or \ddagger\or
   \mathsection\or \mathparagraph\or \|\or \dagger\dagger
   \or \ddagger\ddagger \else\@ctrerr\fi}}

\newlength{\preskip}
\newlength{\postskip}
\makeatother

\newtheoremstyle{theorem}{\preskip}{\postskip}{\itshape}{}{\bfseries}{}
{.5em}{\textbf{\thmname{#1}\thmnumber{ #2} (\thmnote{ #3})}}
\newtheoremstyle{definition}{\preskip}{\postskip}{\normalfont}{0pt}{\bfseries}{}{.5em}{}
\newtheoremstyle{remark}{\preskip}{\postskip}{\normalfont}{0pt}{\bfseries}{}{.5em}{}

\swapnumbers
\theoremstyle{theorem} \newtheorem{thm}{Theorem}[section]
\theoremstyle{theorem} \newtheorem{lem}[thm]{Lemma}
\theoremstyle{theorem} \newtheorem{prop}[thm]{Proposition}
\theoremstyle{theorem} \newtheorem{kor}[thm]{Corollary}
\theoremstyle{definition} 
\theoremstyle{remark} 
\theoremstyle{remark} 
\theoremstyle{definition} 
\theoremstyle{definition} \newtheorem*{ack}{Acknowledgements}
\theoremstyle{remark} 
\theoremstyle{remark} 
\theoremstyle{definition}  \newtheorem{bsp}[thm]{Example}
\theoremstyle{definition}  
\DeclareMathOperator \re {Re}

\DeclareMathOperator \Exp {Exp}

\newcommand{\I}{\mathds{1}}

\newcommand\fa{\qquad \text{for all \ }}

\newcommand{\cadlag}{c\`adl\`ag }

\newcommand\mc[1] {\mathcal{#1}}
\newcommand\mbb[1] {\mathds{#1}}

\newcommand{\eps}{\varepsilon}

\author{%
    Franziska K\"{u}hn\thanks{Institut f\"ur Mathematische Stochastik, Fachrichtung Mathematik, Technische Universit\"at Dresden, 01062 Dresden, Germany, \texttt{franziska.kuehn1@tu-dresden.de}} 
}

\title{Solutions of L\'evy-driven SDEs with unbounded coefficients as Feller processes}

\date{}

\begin{document}

\maketitle

\abstract{\noindent 
Let $(L_t)_{t \geq 0}$ be a $k$-dimensional L\'evy process and $\sigma: \mbb{R}^d \to \mbb{R}^{d \times k}$ a continuous function such that the L\'evy-driven stochastic differential equation (SDE) \begin{equation*} dX_t = \sigma(X_{t-}) \, dL_t, \qquad X_0 \sim \mu \end{equation*} has a unique weak solution. We show that the solution is a Feller process whose domain of the generator contains the smooth functions with compact support if, and only if, the L\'evy measure $\nu$ of the driving L\'evy process $(L_t)_{t \geq 0}$ satisfies \begin{equation*}
	\nu(\{y \in \mbb{R}^k; |\sigma(x)y+x|<r\}) \xrightarrow[]{|x| \to \infty} 0.
\end{equation*}
This generalizes a result by Schilling \& Schnurr \cite{schnurr} which states that the solution to the SDE has this property if $\sigma$ is bounded.
\par\medskip

 \par \medskip

}

\section{Introduction}

Feller processes are a natural generalization of L\'evy processes. They behave locally like L\'evy processes, but -- in contrast to L\'evy processes -- Feller processes are, in general, not homogeneous in space. Although there are several techniques to prove existence results for Feller processes, many of them are restricted to Feller processes with bounded coefficients, i.\,e.\ they assume that the symbol is uniformly bounded with respect to the space variable $x$; see \cite{boett11,ltp} for a survey on known results. In fact, there are only few Feller processes with unbounded coefficients which are well studied, including affine processes and the generalized Ornstein--Uhlenbeck process, cf.\ \cite[Example 1.3f),i)]{ltp} and the references therein. In order to get a better understanding of Feller processes with unbounded coefficients, it is important to find further examples. \par
In the present paper, we investigate under which assumptions the solution to the L\'evy-driven stochastic differential equation (SDE)
 \begin{equation}
	dX_t = \sigma(X_{t-}) \, dL_t \qquad X_0 = x. \label{sde0}
\end{equation}
is a Feller process whose domain of the generator contains the smooth functions with compact support, i.\,e.\ a so-called rich Feller process; here $(L_t)_{t \geq 0}$ is a $k$-dimensional L\'evy process and $\sigma: \mbb{R}^d \to \mbb{R}^{d \times k}$ a continuous function such that the SDE has a unique weak solution. If $\sigma$ is bounded and Lipschitz continuous then this follows from a result by Schilling \& Schnurr \cite{schnurr}. On the other hand, it is known that $(X_t)_{t \geq 0}$ may fail to be a Feller process if $\sigma$ is not bounded. For instance, if $(L_t)_{t \geq 0}$ is a Poisson process with jump intensity $\lambda>0$, then it is not difficult to see that the solution to the SDE \begin{equation*}
	dX_t = - X_{t-} \, dL_t, \qquad X_0 =x,
\end{equation*}
is not a Feller process, see \cite[Remark 3.4]{schnurr} for details. Up to now, it was an open problem under which necessary and sufficient conditions the solution to the SDE \eqref{sde0} is a (rich) Feller process if $\sigma$ is not necessarily bounded. We resolve this question by showing that $(X_t)_{t \geq 0}$ is a rich Feller process if, and only if, the L\'evy measure $\nu_L$ of the driving L\'evy process $(L_t)_{t \geq 0}$ satisfies \begin{equation*}
	\nu_L(\{y \in \mbb{R}^k; \sigma(x)y \in B(-x,r)\}) \xrightarrow[]{|x| \to \infty} 0;
\end{equation*}
here $B(-x,r)$ denotes the open ball of radius $r$ centered at $-x$. The symbol of the Feller process is then given by $q(x,\xi)= \psi(\sigma(x)^T \xi)$ where $\psi$ denotes the characteristic exponent of $(L_t)_{t \geq 0}$. This provides us with a new class of Feller processes with unbounded coefficients. \par \medskip

The following result is our main theorem; the required definitions will be explained in Section~\ref{def}.

\begin{thm} \label{sde-0}
	Let $(L_t)_{t \geq 0}$ be a $k$-dimensional L\'evy process with L\'evy triplet $(b_L,Q_L,\nu_L)$ and characteristic exponent $\psi$. Let $\sigma: \mbb{R}^d \to \mbb{R}^{d \times k}$ be a continuous function satisfying the linear growth condition \begin{equation}
		|\sigma(x)| \leq c (1+|x|), \qquad x \in \mbb{R}^d \label{lin}
	\end{equation}
	for some absolute constant $c>0$. Suppose that the SDE \begin{equation}
		dX_{t} = \sigma(X_{t-}) \, dL_t, \qquad X_0 \sim \mu \label{sde}
	\end{equation}
	admits a unique weak solution for each initial distribution $\mu$ Then the weak solution $(X_t)_{t \geq 0}$ to \eqref{sde} 	is a rich $d$-dimensional Feller process if, and only if, \begin{equation}
		\nu_L(\{y \in \mbb{R}^k; \sigma(x) y \in B(-x,r) \}) \xrightarrow[]{|x| \to \infty} 0 \fa r>0. \label{sde-co}
	\end{equation}
	In this case the symbol $q$ of the Feller process $(X_t)_{t \geq 0}$ is given by $q(x,\xi) := \psi(\sigma(x)^T \xi)$ and $C_c^2(\mbb{R}^d)$ is contained in the domain of the generator. 
\end{thm}

For some classes of L\'evy processes the existence of a unique weak solution to \eqref{sde} can be proved under rather weak regularity assumptions on $\sigma$, typically H{\"o}lder continuity. \par

For Lipschitz continuous functions $\sigma$ we obtain the following corollary.

\begin{kor} \label{sde-1}
	Let $(L_t)_{t \geq 0}$ be a $k$-dimensional L\'evy process with L\'evy triplet $(b_L,Q_L,\nu_L)$ and characteristic exponent $\psi$. If $\sigma: \mbb{R}^d \to \mbb{R}^{d \times k}$ is globally Lipschitz continuous, then the (strong) solution to the SDE \eqref{sde} is a rich $d$-dimensional Feller process if, and only if, \eqref{sde-co} holds. In this case the the symbol $q$ of the Feller process $(X_t)_{t \geq 0}$ is given by $q(x,\xi) := \psi(\sigma(x)^T \xi)$ and $C_c^2(\mbb{R}^d)$ is contained in the domain of the generator.
\end{kor}

Before we prove the results, let us give some intuition what condition \eqref{sde-co} means. For simplicity consider the one-dimensional case, i.\,e.\ $k=d=1$. Then \begin{align*}
	\sigma(x) y \in B(-x,r) &\stackrel{\sigma(x) \neq 0}{\iff} \exists z \in (-r,r): y = -\frac{x}{\sigma(x)}+ \frac{z}{\sigma(x)}.
\end{align*}
Assuming that $|\sigma(x)| \to \infty$ as $|x| \to \infty$, we have $z/\sigma(x) \to 0$ as $|x| \to \infty$, and therefore, heuristically,  \begin{align}
		\nu_L(\{y \in \mbb{R}; \sigma(x) y \in B(-x,r)\}) \approx \nu_L \left( \left\{- \frac{x}{\sigma(x)} \right\} \right) \fa |x| \gg 1. \label{sde-eq5}
\end{align}
There are two cases: \begin{enumerate}
	\item $|x|/|\sigma(x)| \xrightarrow[]{|x| \to \infty} \infty$, i.\,e.\ $\sigma$ is of sublinear growth. In this case, it follows easily from \eqref{sde-eq5} and the dominated convergence theorem that \eqref{sde-co} is automatically satisfied, see Example~\ref{sde-10} for details.
	\item $|x|/|\sigma(x)|$ does not converge to $\infty$ as $|x| \to \infty$. Then \eqref{sde-eq5} shows that \eqref{sde-co} holds if, and only if, $\nu_L$ does not concentrate mass on accumulation points of $-x/\sigma(x)$. This is, in particular, satisfied if the L\'evy measure $\nu_L$ does not have atoms in the closure of the set $\{-x/\sigma(x); |x| \geq R\}$ for $R \gg 1$, see Example~\ref{sde-12} and Example~\ref{sde-14}.
\end{enumerate}

\section{Preliminaries} \label{def}

We consider the Euclidean space $\mbb{R}^d$ endowed with the canonical scalar product $x \cdot y = \sum_{j=1}^d x_j y_j$ and the Borel-$\sigma$-algebra $\mc{B}(\mbb{R}^d)$. We denote as above the open ball of radius $r$ centered at $x$ by $B(x,r)$. We use $C_c^2(\mbb{R}^d)$ to denote the space of twice continuously differentiable functions with compact support and $C_b(\mbb{R}^d)$ is the space of continuous bounded functions $f: \mbb{R}^d \to \mbb{R}$. \par
A $d$-dimensional Markov process $(\Omega,\mc{A},\mbb{P}^x,x \in \mbb{R}^d,X_t,t \geq 0)$ with \cadlag (right-continuous with left-hand limits) sample paths is called a \emph{Feller process} if the associated semigroup $(T_t)_{t \geq 0}$ defined by \begin{equation*}
	T_t f(x) := \mbb{E}^x f(X_t), \quad x \in \mbb{R}^d, f \in \mc{B}_b(\mbb{R}^d) := \{f: \mbb{R}^d \to \mbb{R}; \text{$f$ bounded, Borel measurable}\}
\end{equation*}
has the \emph{Feller property} and $(T_t)_{t \geq 0}$ is \emph{strongly continuous at $t=0$}, i.\,e. $T_t f \in C_{\infty}(\mbb{R}^d)$ for all $C_{\infty}(\mbb{R}^d)$ and $\|T_tf-f\|_{\infty} \xrightarrow[]{t \to 0} 0$ for any $f \in C_{\infty}(\mbb{R}^d)$. Here, $C_{\infty}(\mbb{R}^d)$ denotes the space of continuous functions vanishing at infinity. If the smooth functions with compact support $C_c^{\infty}(\mbb{R}^d)$ are contained in the domain of the generator $(L,\mc{D}(L))$, then we speak of a \emph{rich} Feller process. A result due to von Waldenfels and Courr\`ege, cf.\ \cite[Theorem 2.21]{ltp}, states that the generator $L$ of a rich Feller process is, when restricted to $C_c^{\infty}(\mbb{R}^d)$, a pseudo-differential operator with negative definite symbol: \begin{equation*}
	Lf(x) =  - \int_{\mbb{R}^d} e^{i \, x \cdot \xi} q(x,\xi) \hat{f}(\xi) \, d\xi, \qquad f \in C_c^{\infty}(\mbb{R}^d), \, x \in \mbb{R}^d
\end{equation*}
where $\hat{f}(\xi) := (2\pi)^{-d} \int_{\mbb{R}^d} e^{-ix \xi} f(x) \, dx$ denotes the Fourier transform of $f$ and \begin{equation}
	q(x,\xi) = q(x,0) - i b(x) \cdot \xi + \frac{1}{2} \xi \cdot Q(x) \xi + \int_{\mbb{R}^d \backslash \{0\}} (1-e^{i y \cdot \xi}+ i y\cdot \xi \I_{(0,1)}(|y|)) \, \nu(x,dy). \label{cndf}
\end{equation}
We call $q$ the \emph{symbol} of the rich Feller process $(X_t)_{t \geq 0}$ and $-q$ the symbol of the pseudo-differential operator. For each fixed $x \in \mbb{R}^d$, $(b(x),Q(x),\nu(x,dy))$ is a L\'evy triplet, i.\,e.\ $b(x) \in \mbb{R}^d$, $Q(x) \in \mbb{R}^{d \times d}$ is a symmetric positive semidefinite matrix and $\nu(x,dy)$ a $\sigma$-finite measure on $(\mbb{R}^d \backslash \{0\},\mc{B}(\mbb{R}^d \backslash \{0\}))$ satisfying $\int_{y \neq 0} \min\{|y|^2,1\} \, \nu(x,dy)<\infty$. A point $x \in \mbb{R}^d$ is called \emph{absorbing} if $\mbb{P}^x(X_t=x)=1$ for all $x \in \mbb{R}^d$. If $(X_t)_{t \geq 0}$ is a Feller process and $x \in \mbb{R}^d$ is not absorbing, then the exit time $\tau_r^x := \inf\{t \geq 0; |X_t-x| \geq r\}$ satisfies $\mbb{E}^x \tau_r^x<\infty$ for $r>0$ sufficiently small, cf.\ \cite[Lemma 7.24]{bm2}. \emph{Dynkin's characteristic operator} is the linear operator defined by \begin{equation}
	\mathfrak{L}f(x) := \begin{cases} \lim_{r \to 0} \dfrac{1}{\mbb{E}^x \tau_r^x} \big(\mbb{E}^x f(X_{\tau_r^x})-f(x)\big), & \text{$x$ is not absorbing}, \\ 0, & \text{$x$ is absorbing}, \end{cases} \label{dynkin}
\end{equation}
on the domain $\mc{D}(\mathfrak{L})$ consisting of all functions $f \in \mc{B}_b(\mbb{R}^d)$ such that the limit \eqref{dynkin} exists for all $x \in \mbb{R}^d$. Our standard reference for Feller processes are the monographs \cite{jac1,jac2,jac3,ltp}. \par
A \emph{L\'evy process} $(L_t)_{t \geq 0}$ is a rich Feller process whose symbol $q$ does not depend on $x$. This is equivalent to saying that $(L_t)_{t \geq 0}$ has stationary and independent increments and \cadlag sample paths. The symbol $q=q(\xi)$ (also called \emph{characteristic exponent}) and the L\'evy process $(L_t)_{t \geq 0}$ are related through the L\'evy--Khintchine formula: \begin{equation*}
 \mbb{E}^xe^{i \xi \cdot (L_t-x)} = e^{-t q(\xi)} \fa t \geq 0, \, x,\xi \in \mbb{R}^d.
\end{equation*}
\emph{Weak uniqueness} holds for the \emph{L\'evy-driven stochastic differential equation} (SDE, for short) \begin{equation*}
	dX_t = \sigma(X_{t-}) \, dL_t, \qquad X_0 \sim \mu,
\end{equation*}
if any two weak solutions of the SDE have the same finite-dimensional distributions. We speak of \emph{pathwise uniqueness} if any two strong solutions $(X_t^{(1)})_{t \geq 0}$ and $(X_t^{(2)})_{t \geq 0}$ with $X_0^{(1)}=X_0^{(2)}$ are indistinguishable, i.\,e.\ $\mbb{P}(\forall t \geq 0: X_t^{(1)}=X_t^{(2)})=1$. We refer the reader to the Ikeda \& Watanabe \cite{ikeda} and Protter \cite{protter} for further details. \par
Let $(A,\mc{D})$ be a linear operator with domain $\mc{D} \subseteq \mc{B}_b(\mbb{R}^d)$ and $\mu$ a probability measure on $(\mbb{R}^d,\mc{B}(\mbb{R}^d))$. A $d$-dimensional stochastic process $(X_t)_{t \geq 0}$ with \cadlag sample paths is a \emph{solution to the $(A,\mc{D})$-martingale problem with initial distribution $\mu$} if $X_0 \sim \mu$ and \begin{equation*}
	M_t^f := f(X_t)-f(X_0)- \int_0^t Af(X_s) \, ds, \qquad t \geq 0,
\end{equation*}
is a martingale with respect to the canonical filtration of $(X_t)_{t \geq 0}$ for any $f \in \mc{D}$. The $(A,\mc{D})$-martingale problem is \emph{well-posed} if for any initial distribution $\mu$ there exists a unique (in the sense of finite-dimensional distributions) solution to the $(A,\mc{D})$-martingale problem with initial distribution $\mu$. For a comprehensive study of martingale problems see Ethier \& Kurtz \cite[Chapter 4]{ethier}.

\section{Proofs}

Let us recall the connection between Dynkin's characteristic operator, cf.\ \eqref{dynkin}, and the infinitesimal generator of a Feller process $(X_t)_{t \geq 0}$, cf.\ \cite[Theorem 7.35]{bm2}.

\begin{thm} \label{sde-7}
	Let $(X_t)_{t \geq 0}$ be a Feller process with infinitesimal generator $(L,\mc{D}(L))$ and characteristic operator $(\mathfrak{L},\mc{D}(\mathfrak{L}))$. If $f \in \mc{D}(\mathfrak{L}) \cap C_{\infty}(\mbb{R}^d)$ is such that $\mathfrak{L}f \in C_{\infty}(\mbb{R}^d)$, then $f \in \mc{D}(L)$ and $Lf = \mathfrak{L}f$.
\end{thm}

For the proof of Theorem~\ref{sde-0} we need some auxiliary statements.

\begin{lem} \label{sde-5}
	Let $(q(x,\cdot))_{x \in \mbb{R}^d}$ be a family of continuous negative definite functions (i.\,e.\ functions of the form \eqref{cndf}) with $q(x,0)=0$ and denote by $(b(x),Q(x),\nu(x,dy))_{x \in \mbb{R}^d}$ the associated family of L\'evy triplets. Assume that $x \mapsto q(x,\xi)$ is continuous for all $\xi \in \mbb{R}^d$ and that $q$ is locally bounded in $x$, i.\,e.\ for any compact set $K \subseteq \mbb{R}^d$ there exists $C>0$ such that $|q(x,\xi)| \leq C (1+|\xi|^2)$ for all $x \in K$, $\xi \in \mbb{R}^d$. For the pseudo-differential operator $A$ with symbol $-q$ the following statements are equivalent: \begin{enumerate}
		\item\label{sde-5-ii} $A$ maps $C_c^{\infty}(\mbb{R}^d)$ into $C_{\infty}(\mbb{R}^d)$.
		\item\label{sde-5-iii} $\nu(x,B(-x,r)) \xrightarrow[]{|x| \to \infty} 0$ for any $r>0$.
	\end{enumerate}
	Both conditions are, in particular, satisfied if \begin{equation}
		\lim_{r \to \infty} \sup_{|x| \leq r} \sup_{|\xi| \leq r^{-1}} |q(x,\xi)| = 0. \label{sde-eq9}
	\end{equation}
\end{lem}

For a proof of Lemma~\ref{sde-5} see \cite[Theorem 1.27]{diss} or \cite[Lemma 3.26]{ltp}.

\begin{lem} \label{sde-6}
	Let $\sigma: \mbb{R}^d \to \mbb{R}^{d \times k}$ be a measurable locally bounded mapping and let $(L_t)_{t \geq 0}$ be a L\'evy process with characteristic exponent $\psi$. If the pseudo-differential operator $A$ with symbol $-q(x,\xi) := - \psi(\sigma(x)^T \xi)$ maps $C_c^{\infty}(\mbb{R}^d)$ into $C_b(\mbb{R}^d)$, then the following statements are equivalent. \begin{enumerate}
		\item $(X_t)_{t \geq 0}$ is a weak solution to the SDE \begin{equation}
			dX_t = \sigma(X_{t-}) \, dL_t, \qquad X_0 \sim \mu  \label{sde-eq10}
		\end{equation}
		\item $(X_t)_{t \geq 0}$ is a solution to the $(A,C_c^{\infty}(\mbb{R}^d))$-martingale problem with initial distribution $\mu$. 
	\end{enumerate}
\end{lem}

\begin{proof}
	Let $(X_t)_{t \geq 0}$ be a solution to the SDE \eqref{sde-eq10}. For $r>0$ denote by $\tau_r := \inf\{t \geq 0; |X_t| \geq r\}$ the exit time from the ball $B(0,r)$. Since $(X_t)_{t \geq 0}$ is non-explosive, we have $\tau_r \to \infty$ as $r \to \infty$ almost surely. It follows easily from It\^o's formula that \begin{equation*}
		M_t^{f,r} := f(X_{t \wedge \tau_{r}})-f(X_0) - \int_{[0,t \wedge \tau_r)} Af(X_s) \, ds
	\end{equation*}
	is a $\mbb{P}^{\mu}$-martingale for any $f \in C_c^{\infty}(\mbb{R}^d)$. As $Af \in C_b(\mbb{R}^d)$ and $f \in C_c^{\infty}(\mbb{R}^d) \subseteq C_b(\mbb{R}^d)$, the dominated convergence theorem gives \begin{equation*}
		M_t^{f,r} \xrightarrow[]{r \to \infty} M_t^f := f(X_t)-f(X_0)- \int_0^t Af(X_s) \, ds \quad \text{in $L^1(\mbb{P}^{\mu})$}.
	\end{equation*}
	This implies that $(M_t^f)_{t \geq 0}$ is a martingale for each $f \in C_c^{\infty}(\mbb{R}^d)$. For the proof of the implication (ii) $\implies$ (i) we refer to Kurtz \cite{kurtz}.
\end{proof}

The next result allows us to remove the large jumps from the driving L\'evy process $(L_t)_{t \geq 0}$.

\begin{thm} \label{sde-9}
	Let $(L_t)_{t \geq 0}$ be a $k$-dimensional L\'evy process with L\'evy triplet $(b_L,Q_L,\nu_L)$ and characteristic exponent $\psi$. Let $\sigma: \mbb{R}^d \to \mbb{R}^{d \times k}$ be a continuous function satisfying the linear growth condition $|\sigma(x)| \leq c(1+|x)$. \begin{enumerate}
		\item For any initial distribution $\mu$ the SDE \begin{equation}
			dY_t = \sigma(Y_{t-}) \, dL_t, \qquad Y_0 \sim \mu, \label{sde1}
		\end{equation}
		has a weak solution (which does not explode in finite time).
		\item For fixed $r>0$ denote by $(L_t^{(r)})_{t \geq 0}$ the L\'evy process which is obtained by removing all jumps of modulus larger than $r$ from $(L_t)_{t \geq 0}$, i.\,e.\ a L\'evy process with L\'evy triplet $(b_L,Q_L,\nu_L|_{\overline{B(0,r)}})$. If \eqref{sde1} has a unique weak solution for any initial distribution, then the SDE  \begin{equation}
		dX_t = \sigma(X_{t-}) \, dL_t^{(r)}, \qquad X_0 \sim \mu, \label{sde-trun-3}
	\end{equation}
	has a unique weak solution for any initial distribution $\mu$. 
\end{enumerate}
\end{thm}

\begin{proof}
	Denote by $A$ the pseudo-differential operator with symbol $-\psi(\sigma(x)^T \xi)$, and fix an initial distribution $\mu$. By \cite[Theorem 4.5.4]{ethier} there exists a process $(Y_t)_{t \geq 0}$ which takes values in the one-point compactification of $\mbb{R}^d$ and which is a solution to the  $(A,C_c^{\infty}(\mbb{R}^d))$-martingale problem. It follows from the linear growth condition on $\sigma$ and \cite[Corollary 3.2]{change} that $(Y_t)_{t \geq 0}$ is conservative, i.\,e.\ it does not explode in finite time. Applying Lemma~\ref{sde-6} shows that $(Y_t)_{t \geq 0}$ is a weak solution to \eqref{sde1}, and this proves (i). \par
	It remains to prove (ii). Let $(X_t)_{t \geq 0}$ be a solution to \eqref{sde-trun-3}, and let $(N_j)_{j \geq 1}$ be sequence of independent random variables such that $(N_j)_{j \geq 1}$ and $(X_t)_{t \geq 0}$ are independent and $N_j \sim \Exp(\lambda)$ is exponentially distributed with intensity $\lambda:= \nu_L(\{y; |y|>r\})$ for $j \geq 1$. If we define \begin{equation*}
		\tilde{L}_t^{(r)} := L_t^{(r)} + C_1 \I_{\{t \geq N_1\}} \qquad \qquad \tilde{X}_t := X_t + C_1 X_{N_1-} \I_{\{t \geq N_1\}}
	\end{equation*}
	for a random variable $C_1 \sim \lambda^{-1} \nu_L|_{\overline{B(0,r)}^c}$ which is independent from $(X_t)_{t \geq 0}$ and $(N_j)_{j \geq 1}$, then $(\tilde{X}_{t \wedge N_1})$ is a solution to the SDE \begin{equation*}
		Z_{t \wedge N_1}-Z_0 = \int_0^{t \wedge N_1} \sigma(Z_{t-}) \, d\tilde{L}_t^{(r)}, \qquad Z_0 \sim \mu.
	\end{equation*}
	It follows from It\^o's formula and \cite[Lemma 4.5.16]{ethier} that there exist a solution $(Y_t)_{t \geq 0}$ to the $(A,C_c^{\infty}(\mbb{R}^d))$-martingale problem and a random variable $\tau_1$ taking values in $[0,\infty)$ such that $(\tilde{X}_{t \wedge N_1},N_1)_{t \geq 0}$ and $(Y_{t \wedge \tau_1},\tau_1)_{t \geq 0}$ have the same distribution. By Lemma~\ref{sde-6} $(Y_t)_{t \geq 0}$ is a weak solution to \eqref{sde1}. Set \begin{equation*} 
		\varrho_j(\omega) := k 2^{-j}, \qquad \text{for} \, \, \omega \in \{(k-1)2^{-j} \leq \tau_1 < k 2^{-j}\}, k \geq 1,
	\end{equation*}
	then $\varrho_j \downarrow \tau_1$, and it follows from the dominated convergence theorem that \begin{align*} 
		\mbb{E}^{\mu} (f(Y_{t_1 \wedge \tau_1},\ldots,Y_{t_n \wedge \tau_1}) \mid \tau_1)  
		&=\lim_{j \to \infty} \sum_{k=1}^{\infty} \I_{[(k-1)2^{-j}, k2^{-j})}(\tau_1) \mbb{E}^{\mu}(f(Y_{t_1 \wedge k2^{-j}},\ldots,Y_{t_n \wedge k2^{-j}})
	\end{align*}
	for any bounded continuous function $f: \mbb{R}^n \to \mbb{R}$ and any $0<t_1 < \ldots<t_n$. This shows that the finite-dimensional distributions of $(Y_{t \wedge \tau_1})_{t \geq 0}$ conditioned on $\tau_1$ are uniquely determined by the finite-dimensional distributions of $(Y_t)_{t \geq 0}$. Consequently, \begin{align*}
		\mbb{P}^{\mu}(X_{t_1} \in B_1,\ldots,X_{t_n} \in B_n, t_n<N_1)
		&= \mbb{P}^{\mu}(\tilde{X}_{t_1} \in B_1,\ldots,\tilde{X}_{t_n} \in B_n, t_n<N_1) \\
		&= \mbb{E}^{\mu} \bigg( \mbb{P}^{\mu}(Y_{t_1 \wedge \tau_1} \in B_1,\ldots,Y_{t_n \wedge \tau_1} \in B_n \mid \tau_1) \I_{\{\tau_1<t_n\}} \bigg)
	\end{align*}
	is uniquely determined by the finite-dimensional distributions of $(Y_t)_{t \geq 0}$ and the distribution of $\tau_1 \sim \Exp(\lambda)$ for any Borel sets $B_i$ and $0<t_1<\ldots<t_n$, $n \in \mbb{N}$. We iterate the procedure. Since the shifted process $Z_t := X_{t+N_1+\ldots+N_{j-1}}$ is a weak solution to the SDE \begin{equation*}
		dZ_t = \sigma(Z_{t-}) \, dL_t^{(r)}, \qquad Z_0 \sim \mu_{j}
	\end{equation*}
	we can construct a weak solution $(Y_t)_{t \geq 0}$ to \eqref{sde1}, $Y_0 \sim \mu_j$, and a random variable $\tau_j$ such that $(Z_{t \wedge N_j},N_j)_{t \geq 0} = (Y_{t \wedge \tau_j},\tau_j)_{t \geq 0}$ in distribution, and now we can use the same reasoning as in the first part of the proof. (Note that $\mu_j$ is uniquely determined by $X_t$, $t<N_1+\ldots+N_{j-1}$, since $X_{N_1+\ldots+N_{j-1}} = X_{(N_1+\ldots+N_{j-1})-}$ almost surely.) We conclude that
	 \begin{align*}
			\mbb{P}^{\mu}(X_{t_1} \in B_1,\ldots,X_{t_n} \in B_n, t_n<N_1+\ldots+N_j)
	\end{align*}
	is uniquely determined by finite-dimensional distributions of the unique weak solutions to \eqref{sde2} started at $Y_0 \sim \mu_i$, $i \leq j$. As $N_1+\ldots+N_j \to \infty$ as $j \to \infty$, this proves the weak uniqueness.
\end{proof}

The next result is the key step to prove Theorem~\ref{sde-0}.

\begin{thm} \label{sde-13}
	Let $(L_t)_{t \geq 0}$ and $\sigma$ be as in Theorem~\ref{sde-0}. If \eqref{sde-co} holds, then the solution $(X_t)_{t \geq 0}$ to the SDE \eqref{sde} is a rich Feller process with symbol $q(x,\xi) = \psi(\sigma(x)^T \xi)$, $x,\xi \in \mbb{R}^d$.
	\end{thm}

Note that $(X_t)_{t \geq 0}$ may fail to be a Feller process if \eqref{sde-co} is not satisfied; consider for instance the SDE \begin{equation*}
	dX_t = -X_{t-} \, dN_t, \qquad X_0 = x
\end{equation*}
for a Poisson process $(N_t)_{t \geq 0}$.

\begin{proof}	
	To keep notation simple, we consider only the case $k=d=1$. Since $\sigma$ is at most of linear growth, we can choose $r \in (0,1)$ such that \begin{equation}
		|\sigma(x) y| \leq \frac{1}{2} (|x|+1) \fa x \in \mbb{R}, |y| \leq r.  \label{sde-ast1}
	\end{equation}
	Denote by $(L_t^{(r)})_{t \geq 0}$ the L\'evy process which is obtained from $(L_t)_{t \geq 0}$ by removing all  jumps of modulus larger than $r$. By Theorem~\ref{sde-9}, there exists a unique weak solution to the SDE \begin{equation} 
		dY_t = \sigma(Y_{t-}) \, dL_t^{(r)}, \qquad Y_0 \sim \mu \label{sde-trun-2}
	\end{equation}
	for any initial distribution $\mu$. We split the proof into several steps. \par
	\textbf{Step 1: $(Y_t)_{t \geq 0}$ is a Feller process.} It is well-known that the unique weak solution $(Y_t)_{t \geq 0}$ is a Markov process (the proof works exactly as in the diffusion case, see e.\,g.\ \cite[Section 19.7]{bm2}), see also \cite[Theorem 4.4.2]{ethier}. By  \cite[Lemma 1.4]{ltp}, it suffices to prove that $P_t f(x) := \mbb{E}^x f(Y_t)$ satisfies the following three properties. \begin{enumerate}
		\item $|P_t f(x)-f(x)| \xrightarrow[]{t \to 0} 0$ for all $x \in \mbb{R}$ and $f \in C_{\infty}(\mbb{R})$
		\item $|P_t f(x)| \xrightarrow[]{|x| \to \infty} 0$ for all $t \geq 0$, $f \in C_{\infty}(\mbb{R})$
		\item $x \mapsto P_tf(x)$ is continuous for all $t \geq 0$, $f \in C_{\infty}(\mbb{R})$
	\end{enumerate}
	Since  $(Y_t)_{t \geq 0}$ has \cadlag sample paths, it follows from the dominated convergence theorem that (i) holds. To prove (ii), we are going to show that \begin{equation}
		\lim_{|x| \to \infty} \mbb{P}^x(|Y_t| \leq R) = 0 \fa R>0. \label{cinfty}
	\end{equation}
	If we define a function $f$ by $f(x) := (x^2+1)^{-1}$, then \begin{equation*}
			|f'(x)| \leq 2 |x| f(x)^2 \qquad \text{and} \qquad |f''(x)| \leq 6 f(x)^2, \qquad x \in \mbb{R};
		\end{equation*}
	combining this with \eqref{sde-ast1} and the fact that $\sigma$ grows at most linearly, we find that there exists an absolute constant $C>0$ such that
	\begin{align} \label{gen} \begin{aligned}
		Bf(x) &:= \left(\int_{r<|y|<1} y \, \nu(dy) + b_L \right) \sigma(x) f'(x) + \frac{1}{2} \sigma(x)^2 Q_L f''(x) \\
		&\quad + \int_{0<|y| \leq r} (f(x+\sigma(x)y)-f(x)-f'(x) \sigma(x)y \I_{|y| \leq 1}) \, \nu_L(dy) \end{aligned}
	\end{align}
	satisfies $|Bf(x)| \leq C f(x)$, $x \in \mbb{R}$, for some absolute constant $C>0$. 
	For fixed $R>0$ set $\varrho_R := \inf\{t \geq 0; |Y_t| \leq R\}$ and $Y_t^R := Y_{t \wedge \varrho_R}$. An application of It\^o's formula gives \begin{equation*}
		\mbb{E}^x f(Y_{t \wedge \tau_{\varrho}^x}^R) - f(x) = \mbb{E}^x \left( \int_{[0,t \wedge \tau_{\varrho}^x)} Bf(Y_s^R) \, ds \right);
	\end{equation*}
	here $\tau_{\varrho}^x := \inf\{t \geq 0; |Y_t-x| \geq \varrho	\}$ denotes the exit time from the ball $B(x,\varrho)$. Consequently, \begin{equation*}
		\mbb{E}^x f(Y_{t \wedge \tau_{\varrho}^x}^R) \leq f(x) + C \int_0^t \mbb{E}^x(f(Y_{s \wedge \tau_{\varrho}^x}^R)) \, ds.
	\end{equation*}
	Applying Gronwall's lemma we get \begin{equation*}
		\mbb{E}^x f(Y_{t \wedge \tau_{\varrho}^x}^R) \leq f(x) e^{Ct} \fa t \geq 0, x \in \mbb{R}.
	\end{equation*}
	Since the constant $C$ does not depend on $\varrho$, we obtain from Fatou's lemma \begin{equation*}
		\mbb{E}^x f(Y_t^R) \leq f(x) e^{Ct} \fa t \geq 0, x  \in \mbb{R}.
	\end{equation*}
	This implies by the Markov inequality \begin{align*}
		\mbb{P}^x(|Y_t| \leq R)
		\leq \mbb{P}^x \big(f(Y_t^R) \geq f(R) \big)
		&\leq \frac{1}{f(R)}\mbb{E}^x f(Y_t^R) \\
		&\leq \frac{1}{f(R)} f(x) e^{Ct}.
	\end{align*}
	Obviously, the right-hand side converges to $0$ as $|x| \to \infty$, and this gives \eqref{cinfty}. It remains to prove (iii), i.\,e.\ that $x \mapsto P_t f(x)$ is continuous. Set $g(x) := x^2+1$. Using a very similar reasoning as above, we find $|Bg(x)| \leq C g(x)$ for some absolute constant $C>0$ which implies by It\^o's formula, the optional stopping theorem and Gronwall's lemma that \begin{equation*}
		\mbb{E}^x g(Y_{T \wedge \tau}) \leq g(x) e^{C T} \fa T \geq 0, x \in \mbb{R}
	\end{equation*}
	where $\tau := \inf\{t \geq 0; |Y_t| \geq R\}$. Hence, by the Markov inequality, \begin{align*}
		\mbb{P}^x \left( \sup_{t \leq T} |Y_t| \geq R \right)
		\leq \mbb{P}^x \big(g(Y_{T \wedge \tau}) \geq g(R) \big)
		\leq \frac{1}{g(R)} \mbb{E}^x g(Y_{T \wedge \tau}) 
		&\leq \frac{g(x)}{g(R)} e^{cT}
	\end{align*}
	and the right-hand side converges uniformly (in $x$) on compact sets to $0$ as $R \to \infty$. On the other hand, it follows from the existence of a unique weak solution to \eqref{sde-trun-2}, Lemma~\ref{sde-7} and \cite[Corollary 2.5]{kurtz} that the $(B,C_c^{\infty}(\mbb{R}))$-martingale problem for the operator $B$ defined in \eqref{gen} is well-posed. Now we can apply \cite[Theorem 4.1.16]{jac3} to get tightness and then use exactly the same reasoning as in \cite[Corollary 4.6.4]{jac3} to conclude that $x \mapsto \mbb{E}^x f(Y_t)$ is continuous. 
	Combining the above considerations, we conclude that $(Y_t)_{t \geq 0}$ is a Feller process. From now on we denote by $(M,\mc{D}(M))$ its generator. \par
	\textbf{Step 2: $C_c^{\infty}(\mbb{R}) \subseteq \mc{D}(M)$ and $Mu = Bu$ for any $u \in C_c^{\infty}(\mbb{R})$ with $B$ defined in \eqref{gen}.} 
	Fix $u \in C_c^{\infty}(\mbb{R})$ and denote as usual by $\tau_{\varrho}^x := \inf\{t>0; |Y_t-x| \geq \varrho$\} the exit time from the ball $B(x,\varrho)$. Recall that by It\^o's formula \begin{equation}
	 		\mbb{E}^x u(Y_{t \wedge \tau_{\varrho}^x})-u(x) = \mbb{E}^x \left( \int_{(0,t \wedge \tau_{\varrho}^x)} Bu(Y_s) \, ds \right) \label{sde-eq11}
	 	\end{equation}
	 	for all $x \in \mbb{R}$. Fix a non-absorbing point $x \in \mbb{R}$. By Lemma~\ref{sde-5} and \eqref{sde-co}, we have $Bu \in C_{\infty}(\mbb{R})$. Since $\mbb{E}^x \tau_{\varrho}^x<\infty$ for $\varrho>0$ sufficiently small, it follows from the dominated convergence theorem that \begin{equation*}
	 		\mbb{E}^x u(Y_{\tau_{\varrho}^x})-u(x) = \mbb{E}^x \left( \int_{[0,\tau_{\varrho}^x)} Bu(Y_s) \, ds \right).
	 	\end{equation*}
	 	As $Bf$ is continuous, we get \begin{align*}
	 		\left| \frac{\mbb{E}^x u(Y_{\tau_{\varrho}^x})-u(x)}{\mbb{E}^x \tau_{\varrho}^x} -Bu(x) \right|
	 		&\leq \frac{1}{\mbb{E}^x \tau_{\varrho}^x} \mbb{E}^x \left( \int_{[0,\tau_{\varrho}^x)} |Bu(Y_s)-Bu(x)| \, ds \right) \\
	 		&\leq \sup_{|y-x| \leq \varrho} |Bu(y)-Bu(x)|
	 		\xrightarrow[]{\varrho \to 0} 0.
	 	\end{align*}
	 	If $x \in \mbb{R}$ is absorbing, then \eqref{sde-eq11} gives $Bu(x)=0$. This shows that $u$ is the domain of Dynkin's characteristic operator $\mathfrak{M}$ and $\mathfrak{M}u =Bu$. As $Bu \in C_{\infty}(\mbb{R})$ this implies $u \in \mc{D}(M)$ and $Mu = Bu$, cf.\ Theorem~\ref{sde-7}. \par
	 	\textbf{Step 3: $(X_t)_{t \geq 0}$ is a rich Feller process with symbol $q$.}  We use a perturbation theorem to prove the assertion. Define an operator $N$ by \begin{equation*}
		Nu(x) := \int_{|y|>r} (u(x+\sigma(x)y)-u(x)) \, \nu_L(dy), \qquad x \in \mbb{R}, u \in C_{\infty}(\mbb{R}).
	\end{equation*}
	We claim that $N$ is a bounded linear operator which is dissipative and maps $C_{\infty}(\mbb{R})$ into $C_{\infty}(\mbb{R})$. \emph{Indeed:} Since $\nu_L(B(0,r)^c)<\infty$, the boundedness of $N$ is obvious. Moreover, it is clear from the definition that $N$ is dissipative. 
	Fix $u \in C_{\infty}(\mbb{R})$. Then \begin{equation*}
		|Nu(x)| \leq  \left(|u(x)|+\sup_{|y| \geq R} |u(y)| \right) \int_{|y|>r} \nu_L(dy) + \|u\|_{\infty} \nu_L(\{y \in \mbb{R}; |x+\sigma(x)y| \leq R, |y| >r\}) 
	\end{equation*}
	for any $R>0$. Letting first $|x| \to \infty$ and then $R \to \infty$, it follows from \eqref{sde-co} that $\lim_{|x| \to \infty} |Nu(x)|=0$. Since a straightforward application of the dominated convergence theorem gives the continuity of $Nu$, we conclude $Nu \in C_{\infty}(\mbb{R})$. \par
	Applying \cite[Corollary 1.7.2]{ethier}, we find that there exists a rich Feller process $(Z_t)_{t \geq 0}$ with generator $L:=M+N$. Note that the generator $L$ restricted to $C_c^{\infty}(\mbb{R})$ equals the pseudo-differential operator $A$ with symbol $-q(x,\xi) =- \psi(\sigma(x)^T \xi)$. Since $(Z_t)_{t \geq 0}$ is a solution to the $(A,C_c^{\infty}(\mbb{R}))$-martingale problem and therefore also a weak solution to \eqref{sde}, cf.\ \cite[Corollary 2.5]{kurtz}, it follows from the weak uniqueness of the solution that $X_t \stackrel{d}{=} Z_t$ for all $t \geq 0$ (for a given initial distribution). Hence, $(X_t)_{t \geq 0}$ and $(Z_t)_{t \geq 0}$ have the same semigroup, and this implies that $(X_t)_{t \geq 0}$ is a rich Feller process with symbol $q(x,\xi) = \psi(\sigma(x)^T \xi)$.
\end{proof}

Now we are ready to prove Theorem~\ref{sde-0}.

\begin{proof}[Proof of Theorem~\ref{sde-0}]
	Denote by $A$ the pseudo-differential operator with symbol $-q$. Note that the L\'evy-Khintchine formula shows that the family of L\'evy triplets associated with $q$ is given by \begin{align*}
		b(x) &= \sigma(x) b_L + \int_{|y|<1} \sigma(x) y (\I_{(0,1)}(|\sigma(x)y|)-\I_{(0,1)}(|y|)) \, \nu_L(dy) \\
		Q(x) &= \sigma(x) Q_L \sigma(x)^T  \\
		\nu(x,B) &= \int \I_B(\sigma(x) y) \, \nu_L(dy), \qquad B \in \mc{B}(\mbb{R}^d \backslash \{0\}).
	\end{align*}
	Since the domain of the generator of any rich Feller process contains $C_c^2(\mbb{R}^d)$, cf.\ \cite[Theorem 2.37c)]{ltp}, it suffices to show that the unique weak solution $(X_t)_{t \geq 0}$ is a rich Feller process if, and only if, \eqref{sde-co} holds, and to identify the symbol of the process. \par
 	It is clear from Theorem~\ref{sde-13} that \eqref{sde-co} is sufficient; it remains to prove the necessity of \eqref{sde-co}.
	Suppose that the weak solution $(X_t)_{t \geq 0}$ is a rich Feller process. We claim that $q$ is the symbol of the Feller process $(X_t)_{t \geq 0}$, i.\,e.\ that the generator $L$ satisfies $Af = Lf$ for any $f \in C_c^{\infty}(\mbb{R}^d)$. Denote by $\tau_r^x := \inf\{t \geq 0; |X_t-x| \geq r\}$ the exit time from the ball $B(x,r)$ and fix a non-absorbing point $x \in \mbb{R}^d$. Since $x \mapsto q(x,\xi) = \psi(\sigma(x)^T \xi)$ is continuous for all $\xi \in \mbb{R}^d$ and $q$ is locally bounded, 
	it follows from the dominated convergence theorem that the pseudo-differential operator $A$ maps $C_c^{\infty}(\mbb{R}^d)$ into $C(\mbb{R}^d)$. Therefore an application of It\^o's formula and the fundamental theorem of calculus yields \begin{equation*}
		\lim_{r \to 0} \frac{1}{\mbb{E}^x \tau_r^x} (\mbb{E}^x f(X_{\tau_r^x})-f(x)) = Af(x) \fa f \in C_c^{\infty}(\mbb{R}^d).
	\end{equation*}
	On the other hand, $C_c^{\infty}(\mbb{R}^d)$ is contained in the domain of the generator $(L,\mc{D}(L))$ of the Feller process $(X_t)_{t \geq 0}$, and therefore by Dynkin's formula
 	\begin{equation*}
 		\lim_{r \to 0} \frac{1}{\mbb{E}^x \tau_r^x} (\mbb{E}^x f(X_{\tau_r^x})-f(x)) = Lf(x) \fa f \in C_c^{\infty}(\mbb{R}^d).
 	\end{equation*}
 	Hence, $Lf(x) = Af(x)$ for any non-absorbing point $x \in \mbb{R}^d$. If $x \in \mbb{R}^d$ is absorbing, then it is not difficult to see that $Af(x)=0=Lf(x)$. Hence, $Af=Lf$ for any $f \in C_c^{\infty}(\mbb{R}^d)$. This shows that $q$ is the symbol of the Feller process $(X_t)_{t \geq 0}$. By Lemma~\ref{sde-5}, $Af = Lf \in C_{\infty}(\mbb{R}^d)$ implies \eqref{sde-co}.   \par 
\end{proof}

\begin{proof}[Proof of Corollary~\ref{sde-1}]
	It is well known (see e.\,g.\ \cite[Theorem IV.9.1]{ikeda}) that there exists a (pathwise) unique strong solution to the SDE \eqref{sde} if $\sigma$ is Lipschitz continuous. Since any strong solution is also a weak solution and pathwise uniqueness implies uniqueness in law (see e.\,g.\ \cite[Corollary 140]{situ}), there exists a unique (in law) weak solution to \eqref{sde} for any $r \in (0,\infty]$. Now the claim follows from Theorem~\ref{sde-0}.
\end{proof}

\section{Examples}

In this section we present some illustrating examples. 

\begin{bsp}[Sublinear growth] \label{sde-10}
	Let $(L_t)_{t \geq 0}$ be a $k$-dimensional L\'evy process with L\'evy triplet $(b_L,Q_L,\nu_L)$ and characteristic exponent $\psi$ and let $\sigma: \mbb{R}^d \to \mbb{R}^{d \times k}$ be a Lipschitz continuous function. If $\sigma$ is of sublinear growth, i.\,e.\  \begin{equation*}
		\lim_{|x| \to \infty} \frac{|\sigma(x)|}{|x|} = 0,
	\end{equation*}
	then the strong solution to the SDE \begin{equation}
		dX_t = \sigma(X_{t-}) \, dL_t, \qquad X_0 = x, \label{sde2}
	\end{equation}
	is a rich $d$-dimensional Feller process with symbol $q(x,\xi) = \psi(\sigma(x)^T  \xi)$.
\end{bsp}

Let us mention that this statement can be also deduced from a result by B\"{o}ttcher \cite{boett11}.

\begin{proof}[Proof of Example~\ref{sde-10}]
	We show that any function $\sigma$ of sublinear growth satisfies \eqref{sde-co}. By the triangle inequality, we have \begin{align*}
		\nu_L(\{y \in \mbb{R}^k; \sigma(x) y \in B(-x,r)\})
		&= \nu_L \left( \left\{y \in \mbb{R}^k; \frac{|\sigma(x) y+x|}{|x|} < \frac{r}{|x|} \right\} \right) \\
		&\leq \nu_L \left( \left\{y \in \mbb{R}^k; 1- \frac{|\sigma(x)|}{|x|} |y| < \frac{r}{|x|} \right\} \right).
	\end{align*}
	For any fixed $r>0$ and $\eps>0$ we can choose $R>0$ sufficiently large such that $r/|x| \leq \eps$ and $|\sigma(x)|/|x| \leq \eps$ for all $|x| \geq R$. Hence, \begin{align*}
		\nu_L(\{y \in \mbb{R}^k; \sigma(x) y \in B(-x,r)\})
		&\leq \nu_L \left( \left\{y  \in \mbb{R}^k; 1- \eps |y| \leq \eps \right\} \right) \\
		&= \nu_L \left( \mbb{R}^k \backslash B \left(0, \eps^{-1}-1 \right) \right)
		\xrightarrow[]{\eps \to 0} 0. \qedhere
	\end{align*}
\end{proof}

If the driving L\'evy process $(L_t)_{t \geq 0}$ is a one-dimensional symmetric $\alpha$-stable process, we obtain the following stronger result.

\begin{bsp} \label{sde-11}
	Let $(L_t)_{t \geq 0}$ be a one-dimensional symmetric $\alpha$-stable L\'evy process for some $\alpha \in (0,2]$, i.\,e.\ a L\'evy process with characteristic exponent $\psi(\xi) = |\xi|^{\alpha}$, $\xi \in \mbb{R}$. Let $\sigma: \mbb{R} \to \mbb{R}$ be a continuous mapping such that the SDE \eqref{sde2} has a unique weak solution (taking values in $\mbb{R}^d$) for any initial distribution $\mu$. Then the unique weak solution $(X_t)_{t \geq 0}$ is a rich Feller process if, and only if, \begin{equation}
		\lim_{|x| \to \infty} \frac{|\sigma(x)|^{\alpha}}{|x|^{1+\alpha}} = 0. \label{alpha-grow}
	\end{equation}
	In this case, the symbol $q$ of the Feller process $(X_t)_{t \geq 0}$ is given by $q(x,\xi) = |\sigma(x)|^{\alpha} \, |\xi|^{\alpha}$.
\end{bsp}

Clearly, the growth condition~\eqref{alpha-grow} is, in particular, satisfied if $\sigma$ is at most of linear growth.

\begin{proof}[Proof of Example~\ref{sde-11}]
	Fix $r>0$. Since \begin{equation*}
		\int_{y \neq 0} \I_{B(-x,r)}(\sigma(x) y) \frac{1}{|y|^{1+\alpha}} \, dy = |\sigma(x)|^{\alpha} \int_{z \neq 0} \I_{B(-x,r)}(z) \frac{1}{|z|^{1+\alpha}} \, dz
	\end{equation*}
	for any $|x|>r$, it follows easily from the fact that \begin{equation*}
		\int_{z \neq 0} \I_{B(-x,r)}(z) \frac{1}{|z|^{1+\alpha}} \, dz \asymp |x|^{-\alpha-1}, \qquad |x| \gg 1
	\end{equation*}
	that \eqref{sde-co} is equivalent to \eqref{alpha-grow}. Now the assertion follows from Theorem~\ref{sde-0}.
\end{proof}

It is known that the SDE \eqref{sde2} driven by a one-dimensional symmetric $\alpha$-stable L\'evy process has a unique weak solution if $\sigma$ is continuous and one of the following two conditions is satisfied.	
\begin{enumerate}
		\item (cf.\ Zanzotto \cite{zan02}) $\alpha \in (1,2]$ and \begin{equation}
		\{x \in \mbb{R}; \sigma(x) = 0\} = \left\{x \in \mbb{R}; \forall \delta \in (0,1): \int_{x-\delta}^{x+\delta} |\sigma(y)|^{-\alpha} \, dy=\infty \right\} \label{iso}
	\end{equation}
		\item (cf.\ K\"uhn \cite{change}) $\alpha \in (0,2]$, $\inf_x \sigma(x)>0$.
	\end{enumerate}
Note that \eqref{iso} is satisfied for any continuous function $\sigma$ such that $\inf_{x \in \mbb{R}} \sigma(x)>0$, and therefore the first condition is for $\alpha \in (1,2]$ more general than the second one. \par
Example~\ref{sde-11} shows, in particular, that for a one-dimensional symmetric $\alpha$-stable L\'evy process $(L_t)_{t \geq 0}$, $\alpha \in (1,2]$,  the (weak) solution to the SDE \begin{equation*}
	dX_t = |X_{t-}|^{\beta} \, dL_t, \qquad X_0 = x,
\end{equation*}
is a rich Feller process with symbol $q(x,\xi) = |x|^{\beta \alpha} \, |\xi|^{\alpha}$ for all $\beta \in [1/\alpha,1)$. \par \medskip

The next example discusses generalized Ornstein--Uhlenbeck processes which have been studied by Behme \& Lindner \cite{behme}. 
		
\begin{bsp}[Generalized Ornstein--Uhlenbeck process] \label{sde-12}
	Let $L_t=(U_t,V_t)$ be a two-di\-men\-sion\-al L\'evy process with characteristic exponent $\psi$ and L\'evy triplet $(b_L,Q_L,\nu_L)$. The generalized Ornstein--Uhlenbeck process \begin{equation*}
		dX_t = X_{t-} \, dU_t  + dV_t, \qquad X_0 = x,
	\end{equation*}
	is a rich Feller process (with symbol $q(x,\xi) = \psi((x,1)^T \xi)$) if, and only if, \begin{equation}
		\nu_L(\{-1\} \times \mbb{R})=0. \label{gou}
	\end{equation}	
\end{bsp}

Behme \& Lindner \cite{behme} proved that the generalized Ornstein--Uhlenbeck process is a rich Feller process if \eqref{gou} holds. Our proof is not only considerably shorter, but also shows that \eqref{gou} is, in fact, a necessary and sufficient condition.

\begin{proof}[Proof of Example~\ref{sde-12}]
	Set $\sigma(x) := (x,1)$ for $x \in \mbb{R}$ and \begin{align*}
		A(x) 
		:= \{y \in \mbb{R}^2; |\sigma(x)y + x| \leq r\} 
		&= \{y \in \mbb{R}^2; |xy_1+y_2+x| \leq r\} \\
		&= \left\{y \in \mbb{R}^2; \left| y_1 + \frac{y_2}{x} + 1 \right| \leq \frac{r}{|x|} \right\}.
	\end{align*}
 	Since $\nu_L$ is a $\sigma$-finite measure on $\mbb{R}^2 \backslash \{0\}$ and $\I_{A(x)}(y) \xrightarrow[]{|x| \to \infty} \I_{\{y_1=-1\}}$, the dominated convergence theorem yields $\nu_L(A(x)) \xrightarrow[]{|x| \to \infty} \nu_L(\{-1\} \times \mbb{R})$. Consequently, \eqref{sde-co} holds if, and only if, $\nu_L(\{-1\} \times \mbb{R})=0$. 
	\end{proof}	

Using a very similar reasoning, we obtain the following result on solutions of linear SDEs.

\begin{bsp}[Linear SDE] \label{sde-14}
	Let $(L_t)_{t \geq 0}$ be a $d$-dimensional L\'evy process with characteristic exponent $\psi$ and L\'evy triplet $(b_L,Q_L,\nu_L)$. Then for any fixed $c \in \mbb{R}^d$ the solution to the linear SDE \begin{equation*}
		dX_t = c X_{t-} dL_t, \qquad X_0 = x,
	\end{equation*}
	is a rich Feller process (with symbol $q(x,\xi) := \psi(x \xi c)$, $x,\xi \in \mbb{R}$), if, and only if, \begin{equation*}
		\nu_L\left( \left\{y \in \mbb{R}^d; c \cdot y = -1 \right\} \right)=0.
	\end{equation*}
\end{bsp}

Example~\ref{sde-16} shows that Theorem~\ref{sde-0} may fail to hold if the coefficient $\sigma$ does not satisfy the linear growth condition \eqref{lin}.

\begin{bsp} \label{sde-16}
	The ordinary differential equation \begin{equation*}
		dX_t = - X_t^3 \, dt, \qquad X_0 =  x
	\end{equation*}
	has a unique solution which is given by \begin{equation*}
		X_t = \frac{x}{\sqrt{1+2tx^2}}, \qquad t \geq 0, x \in \mbb{R}.
	\end{equation*}
	Although \eqref{sde-co} is trivially satisfied, the process $(X_t)_{t \geq 0}$ is \emph{not} a Feller process since the associated semigroup $(T_t)_{t \geq 0}$ does not satisfy the Feller property: \begin{equation*}
		\lim_{|x| \to \infty} T_t f(x) = f \left( \frac{1}{\sqrt{2t}} \right),
	\end{equation*}
	cf.\ \cite[Example 2.26b)]{ltp} and \cite[Remark 2.12]{cast2}. 
\end{bsp}

\begin{ack}
	I would like to thank Bj\"{o}rn B\"{o}ttcher, Alexei Kulik and Ren\'e Schilling for their helpful comments and suggestions. An earlier version of this paper relied on a result \cite[Theorem 2.5]{cast1} by van Casteren; many thanks to Tristan Haugomat for pointing out that this result is wrong and that a counterexample can be found in \cite[Remark 2.12]{cast2} (see Example~\ref{sde-16}). The error in van Casteren's \cite{cast1} proof is in Proposition 2.2 where it is claimed that $x \mapsto \mbb{E}^x u(X_s)$ vanishes at $\infty$ for any $u \in C_{\infty}(\mbb{R}^d)$; as Example~\ref{sde-16} shows this is, in general, not correct. However, a close look at his proof reveals that a weaker statement than the Feller property holds true: For any $u \in C_{\infty}(\mbb{R}^d)$ and $t \geq 0$ there exists a constant $c \in \mbb{R}$ such that $\lim_{|x| \to \infty} \mbb{E}^x u(X_t)=c$. The current version of this paper does not rely on van Casterens results.
\end{ack}

\end{document}